\documentclass{article}[12pt]

\usepackage{amsmath}
\usepackage{amsfonts}
\usepackage{amsthm}
\usepackage{graphicx}
\usepackage[latin1]{inputenc}
\usepackage{color}

\newtheorem{remark}{Remark}[section]
\newtheorem{theorem}[remark] {Theorem}
\newtheorem{prop}[remark]{Proposition}
\newtheorem{lemma}[remark]{Lemma}
\newtheorem{cor}[remark]{Corollary}

\newcommand{\R}{\mathbb R}

\begin{document}
\large \title{Fluctuation Theorems for Synchronization
  of\\ Interacting P\'olya's urns} 

\author{Irene Crimaldi\footnote{IMT Institute for Advanced Studies
    Lucca, Piazza San Ponziano 6, I-55100 Lucca, Italy, e-mail:
    irene.crimaldi@imtlucca.it (corresponding author)}, Paolo Dai
  Pra\footnote{Dipartimento di Matematica, Universit\`a degli Studi di
    Padova, Via Trieste 63, I-35121 Padova, Italy, e-mail:
    daipra@math.unipd.it}, Ida Germana Minelli\footnote{Dipartimento
    di Ingegneria e Scienze dell'Informazione e Matematica,
    Universit\`a degli Studi dell'Aquila, Via Vetoio (Coppito 1),
    I-67100 Coppito (AQ), Italy, e-mail: ida.minelli@dm.univaq.it}}

\maketitle

\abstract{We consider a model of $N$ two-colors urns in which the
  reinforcement of each urn depends also on the content of all the
  other urns. This interaction is of mean-field type and it is tuned
  by a parameter $\alpha \in [0,1]$; in particular, for $\alpha = 0$
  the $N$ urns behave as $N$ independent P\'olya's urns. For $\alpha >
  0$ urns synchronize, in the sense that the fraction of balls of a
  given color converges a.s. to the same (random) limit in all urns.
  In this paper we study fluctuations around this synchronized
  regime. The scaling of these fluctuations depends on the parameter
  $\alpha$. In particular the standard scaling $t^{-1/2}$ appears only
  for $\alpha>1/2$. For $\alpha \geq 1/2$ we also determine the limit
  distribution of the rescaled fluctuations. We use the notion of
  stable convergence, which is stronger than convergence in
  distribution.  \\


}

\section{Introduction}

In this paper we continue the study of synchronization for a model of
interacting P\'olya's urn that has been introduced in
\cite{daipra-2014}. This study is motivated by the attempt of
understanding the role of reinforcement in synchronization
phenomena. Here the word {\em synchronization} is meant in the wide
sense of ``coherent behavior of the majority'', that could be time
stationary or time periodic.  Experimental results in the context of
cellular and neuronal systems have stimulated the formulation and the
analysis of stylized stochastic models that could reveal the origin of
such phenomena, in particular in systems that are not time-reversible
(see e.g. \cite{GB, perth1}). The wide majority of the models proposed
consists of time-homogeneous Markov processes; with this choice,
long-time correlations and {\em aging} are usually ruled out.

One way of breaking time-homogeneity and adding memory to the dynamics
consists in introducing a reinforcement mechanism. The most popular
stylized model in this context is the P\'olya's urn model. In the
simplest version, the model consists of an urn which contains balls of
two different colors (for example, at time $t=0$, $a\geq 1$ red and
$b\geq 1$ black balls). At each discrete time a ball is drawn out and it
is replaced in the urn together with another ball of the same color.
Let $Z_t$ be the fraction of red balls at time $t$, namely, the
conditional probability of drawing a red ball at time $t+1$, given the
fraction of the red balls at time $t$.  A well known result (see for
instance~\cite{klenke2007probability} or \cite{mah-2008}) states that
$(Z_t)_t$ is a bounded martingale and in particular $\lim_{t\rightarrow
  \infty} Z_t=Z_\infty$ a.s., where $Z_\infty$ has Beta distribution
with parameters $a$ and $b$.

In \cite{daipra-2014} an interacting version of this model is
formulated. Consider a set of $N>1$ P\'olya's urns and introduce a
``mean field interaction'' among them:
\begin{itemize}
\item
at time $0$, each urn contains $a\geq 1$ red and $b\geq 1$ black balls;
\item
at each time $t+1$, a new ball is introduced in each urn and, given
the fraction $Z_t(i)$, for $1\leq i\leq N$, of red balls in each urn
$i$ at time $t$, the ball added in urn $j$ is, independently of what
happens for all the other urns, red (otherwise it is black) with
conditional probability $\alpha Z_t+(1-\alpha)Z_t(j)$, where $Z_t$
is the total fraction of red balls in the system at time $t$,
i.e. $Z_t=\frac{1}{N}\sum_{i=1}^N Z_t(i)$, and $\alpha\in [0,1]$.
\end{itemize}
 
The case $\alpha=0$ corresponds to $N$ independent copies of the
classical P\'olya's urn described above. Thus, for $1\leq j\leq N$,
each proportion $Z_t(j)$ converges, as $t \rightarrow +\infty$, to
i.i.d. random variables, whose distribution is Beta with parameters
$a,\, b$. As soon as $\alpha >0$, some basic properties of the
P\'olya's urn model is lost: in particular the sequences
$\big(Z_t(j)\big)_t$ are not martingales (although $(Z_t)_t$ is a
martingale), and the sequences of the colors drawn are no more
exchangeable. It is shown in \cite{daipra-2014}, that 
$D_t(j):=Z_t(j)-Z_t$ converges to zero almost surely and in $L^2$; as a
consequence, all the fractions $Z_t(j)$ converge a.s. to the {\em same
  limit} $Z$. We refer to this phenomenon as {\em almost sure
  synchronization} of the system of interacting urns. It is relevant
to note that this is not a macroscopic or {\em thermodynamic} effect:
the number of urns $N$ is kept fixed. The ``phase transition'' from
disorder ($\alpha = 0$) to synchronization ($\alpha >0$) is not a
consequence of the large scale of the system but of the long memory
caused by the reinforcement.

In the present paper we analyze in detail the {\em fluctuations} of
$D_t(j)= Z_t(j)-Z_t$ around zero as $t\to +\infty$. The rate of
convergence to zero in $L^2$-norm has been already analyzed in
\cite{daipra-2014}, revealing an interesting scaling for certain
values of the interaction parameter $\alpha$: $D_t(j)$ scales as
$t^{-1/2}$ for $\frac{1}{2}< \alpha \leq 1$, as $t^{-1/2}
\sqrt{\ln(t)}$ for $\alpha = \frac{1}{2}$, and as $t^{-\alpha}$ for
$0<\alpha<\frac{1}{2}$. In this paper, we obtain limit theorems for
the rescaled fluctuations: for $\alpha \geq \frac{1}{2}$, they
converge in distribution, as $t \to +\infty$, to a mixture of centered
Gaussian distributions, whose random variance is an explicit function
of the limit random variable $Z$; for $0<\alpha<\frac{1}{2}$, the
rescaled fluctuations converge {\em almost surely} to a nonzero real
random variable. Indeed, in the case $\alpha \geq \frac{1}{2}$, we
prove {\em stable convergence} (see e.g. \cite{cri-let-pra-2007}),
which is strictly stronger than convergence in distribution
(basic definitions and results on this form of convergence will be
recalled later on). 

We note that the scaling phenomenon in our model resembles the one
observed in a single Friedman's urn (see \cite{fr65,
  fr49}). Friedman's urn model is a modification of P\'olya's urn in
which, after each drawing, one adds in the urn $A>0$ balls of the
drawn color, and $B>0$ balls of the other color. In this model the
proportion $Z_t$ of red balls converges a.s. to $\frac{1}{2}$ and
fluctuations around this limit exhibit an analogous scaling depending
whether the parameter $\alpha:= 1-\frac{A-B}{A+B}$ is smaller or equal
or bigger than $1/2$. (However, we point out that the limit
distributions obtained by Friedman for $\alpha\geq 1/2$ are simple
Gaussian distributions and not mixtures of Gaussian distributions.)
The analogy can be explained by a similar equation for the conditional
probability of drawing a red ball at time $t+1$. Indeed, our arguments
for the proofs in this paper could be used to strengthen some of the
results in \cite{fr65}. We will not elaborate further on this point.

We also mention that central limit theorems for a single randomly
reinforced P\'olya's urn has been recently given in \cite{aletti-2009,
  ber-cri-pra-rig-2011, ber-cri-pra-rig-2010, crimaldi-2009}.
Finally, we remark that models of interacting urns have been
considered by various authors (see e.g. \cite{pemantle-2007} for a
general survey on random processes with reinforcement). However, they
are different from the one studied in this paper. In particular, in
\cite{marsili-1998}, the authors introduced a model which describes a
system of interacting agents, modeled by urns arranged on a lattice,
subject to perturbations and occasionally break down. Indeed, each urn
contains $b$ black and one white balls; at each time, a ball is drawn
from a certain urn and, if it is white, then a new white ball is added
in the urn, while, if it is black, then the urn comes back to the
initial composition and, for each white ball previously present in the
urn, a similar attempt to add a white ball is made on a randomly
chosen nearest neighbour urn.  In \cite{cirillo-2012}, the authors
also consider a system of interacting urns on a lattice, representing
agents subject to defaults. The reinforcement matrix is not only a
function of time (time contagion), but also of the behavior of the
neighboring urns (spatial contagion) and of a random component. In
\cite{paganoni}, a countable collection of interacting urns is
considered in which, at each time, a ball is sampled from each urn and
a random number of new balls of the same color of the extracted one is
introduced in the urn together with the drawn ball. The distribution
of the reinforcement for urn $j$ depends on the colors extracted from
the urns with index $i\neq j$ and on an independent random
factor. Urns with a mean-field interaction have been considered in
\cite{launay-2012, launay}, but with a reinforcement scheme different
from ours: their main results are proven when the probability of
drawing a ball of a certain color is proportional to the exponential
of the number of balls of that color, rather than to the number of
balls of that color, leading to a quite different synchronization
picture.

The paper is organized as follows. In Section~\ref{prelimin} we
formally introduce the model and recall the needed facts concerning
stable convergence. In Section~\ref{results} we give and discuss the
statement of our main results, whose proofs are postponed to
Section~\ref{proofs}. The paper is enriched with an appendix which
contains some useful auxiliary results.

\section{Setting, notation and preliminaries}
\label{prelimin}

We consider the interacting system introduced in
\cite{daipra-2014}. More precisely, we have the following system of
$N>1$ (with $N$ finite) P\'olya's urns on a probability space
$(\Omega, {\cal A}, P)$, with ``mean field interaction'':
\begin{itemize}
\item at time $t=0$, each urn contains $a\geq 1$ red balls and $b\geq
  1$ black balls;
\item at time $t+1$, a new ball is added in each urn as follows: given
  the proportion $Z_t(i)$, for $1\leq i\leq N$, of red balls in each
  urn at time $t$, the ball added in urn $j$ is, independently of
  what happens in all the urns with $i\neq j$, red (otherwise it is
  black) with conditional probability
\begin{equation}\label{rel-nostra}
\alpha Z_t + (1-\alpha)Z_t(j)
\end{equation}
where $Z_t=\frac{1}{N}\sum_{i=1}^N Z_t(i)$ is the total proportion
of red balls in the system at time $t$ and $\alpha$ is a parameter in
$[0,1]$ which tunes the interaction among the urns ($\alpha=0$
corresponds to $N$ independent P\'olya's urns and $\alpha=1$
corresponds to the admissible maximum level of interaction).
\end{itemize}

\indent For all the sequel, we set $m=a+b$ and $D_t(j)=(Z_t(j)-Z_t)$
(for simplicity, sometimes we omit the index $j$, i.e. we set
$D_t=D_t(j)$, when $j$ is fixed) and, for each $i=1,\dots, N$, we
denote by $I_{t+1}(i)$ the indicator function of the event $\{$red
ball for urn $i$ at time $t+1\}$\footnote{It corresponds to the
  random variable $Y_t(i)$ in \cite{daipra-2014}.} and we
define the (increasing) filtration ${\cal F}=({\cal F}_t)_t$ as 
$$
{\cal F}_0=\{\Omega, \emptyset\}\quad\hbox{and}\quad
{\cal F}_t=\sigma(I_k(i):\, i=1,\dots, N,\; 1\leq k\leq t).
$$

\indent It is easy to verify that $(Z_t)$ is an $\cal F$-martingale
which converges almost surely to a random variable $Z$. Furthermore,
in \cite{daipra-2014}, authors verified that each $(Z_t(j))$ is an
$\mathcal F$-quasi-martingale (a martingale when $\alpha=0$) and, for
$\alpha>0$, they proved the almost sure synchronization, i.e.
$$
Z_t(j)\stackrel{a.s.}\longrightarrow Z\qquad\forall j\in\{1,\dots,N\}.
$$ 

These last two properties are a consequence of the relation
\begin{equation}\label{rel-4}
E[D_t(j)^2]=E[(Z_t(j)-Z_t)^2]\sim d(\alpha)
\begin{cases}
t^{-2\alpha}\quad&\mbox{when } 0<\alpha<1/2\\
t^{-1}\ln(t) \quad&\mbox{when } \alpha=1/2\\
t^{-1} \quad&\mbox{when } 1/2<\alpha\leq 1\\
\end{cases}
\end{equation}
which holds true for each fixed urn $j$ and $t\to
+\infty$\footnote{The symbol $d(\alpha)$ denotes a constant in
  $(0,+\infty)$ and we use the notation $a_t\sim b_t$ when $\lim_{t\to
    +\infty} a_t/b_t=1$.}.\\

\indent We finally recall some simple relations for a fixed $j$, that
will be useful in the following proofs:
$$
Z_{t+1}(j)-Z_t(j)=\frac{I_{t+1}(j)-Z_t(j)}{m+t+1},\qquad
Z_{t+1}-Z_t=\frac{(\sum_{i=1}^N I_{t+1}(i)/N)-Z_t}{m+t+1},
$$ 
$$
E[I_{t+1}(j)\,|\,{\cal F}_t]=\alpha Z_t+(1-\alpha)Z_t(j),
\qquad
E\left[{\textstyle\sum_{i=1}^N I_{t+1}(i)/N}\,|\,{\cal F}_t\right]=Z_t,
$$
and
$$
E[Z_{t+1}(j)\,|\,{\cal F}_t]-Z_t(j)=\frac{\alpha}{m+t+1}(Z_t-Z_t(j))=
\frac{-\alpha D_t}{m+t+1}.
$$

\indent We conclude this section with a brief review on stable
convergence.

\subsection{Stable convergence and its generalizations}

Stable convergence has been introduced by R\'enyi in \cite{renyi-1963}
and subsequently investigated by various authors, e.g.
\cite{aldous-1978, cri-let-pra-2007, feigin-1985, jacod-1981,
  peccati-2008}. It is a strong form of convergence in distribution,
in the sense that it is intermediate between the simple convergence in
distribution and the convergence in probability. We recall here some
basic definitions. For more details, we refer the reader to
\cite{cri-let-pra-2007, hall-1980} and the references therein.

Let $(\Omega, {\cal A}, P)$ be a probability space, and let $S$ be a
Polish space, endowed with its Borel $\sigma$-field. A {\em kernel} on
$S$, or a random probability measure on $S$, is a collection
$K=\{K(\omega):\, \omega\in\Omega\}$ of probability measures on the
Borel $\sigma$-field of $S$ such that, for each bounded Borel real
function $f$ on $S$, the map
$$
\omega\mapsto 
K(f)(\omega)=\int f (x)\, K(\omega)(dx) 
$$
is $\cal A$-measurable.

On $(\Omega, {\mathcal A},P)$, let $(Y_t)$ be a sequence of $S$-valued
random variables and let $K$ be a kernel on $S$. Then we say that
$Y_t$ converges {\em stably} to $K$, and we write
$Y_t\stackrel{stably}\longrightarrow K$, if
$$
P(Y_t \in \cdot \,|\, H)\stackrel{weakly}\longrightarrow 
E\left[K(\cdot)\,|\, H \right]
\qquad\hbox{for all } H\in{\cal A}\; \hbox{with } P(H) > 0.
$$

Clearly, if $Y_t\stackrel{stably}\longrightarrow K$, then $Y_t$
converges in distribution to the probability distribution
$E[K(\cdot)]$. Moreover the convergence in probability 
of $Y_t$ to a random variable $Y$ is equivalent to the stable convergence 
of $Y_t$ to a special kernel, which is the Dirac kernel $K=\delta_Y$.  

We next mention two generalizations of the notion of stable
convergence: a strong form of stable convergence, introduced and
studied in \cite{cri-let-pra-2007}, and the almost sure conditional
convergence, introduced and studied in \cite{crimaldi-2009}, that will
be used later on.

For each $t$, let ${\cal F}_t$ be a sub-$\sigma$-field of $\cal
A$. We say that $Y_t$ converges to $K$ {\em stably in the strong sense}, with
respect to the sequence ${\cal F}=({\cal F}_t)$ (called conditioning system), 
if 
$$
E\left[f(Y_t)\,|\,{\cal F}_t\right]\stackrel{P}\longrightarrow K(f)
$$
for each bounded continuous real function $f$ on $S$. 

A strengthening of the stable convergence in the strong sense is the
following. If $K_t$ denotes a version of the conditional distribution
of $Y_t$ given ${\mathcal F}_t$, we say that $Y_t$ converges to $K$ in
the sense of the {\em almost sure conditional convergence}, with
respect to ${\cal F}$, if, for almost every $\omega$ in $\Omega$, the
probability measure $K_t(\omega)$ converges weakly to $K(\omega)$ and
so
$$
E\left[f(Y_t)\,|\,{\cal F}_t\right]\stackrel{a.s.}\longrightarrow K(f)
$$ 
for each bounded continuous real function $f$ on $S$. 

\section{Results}
\label{results}

This section is devoted to the statement of our main results, whose
proofs are postponed to Section \ref{proofs}. The first result is a
Central Limit Theorem for the total fraction $Z_t$ of red balls in the
whole system. Note that the fluctuations of $Z_t$ around $Z$ scale as
$t^{-1/2}$ for all values of $\alpha$.

\begin{theorem}\label{th-media-proporz}
We have 
$$
\sqrt{t}(Z_t-Z)\stackrel{stably}\longrightarrow 
{\cal N}\left(0,\frac{1}{N}(Z-Z^2)\right)
$$ 
The above sequence also converges in the sense of the almost sure
conditional convergence with respect to
the filtration $\mathcal F$.
\end{theorem}

As observed in \cite{aletti-2009}, an advantage of having the almost
sure conditional convergence is the fact that it allows to prove that
the distribution of the limit random variable $Z$ has no point masses
in $(0,1)$. As a consequence of this, we get that the limit Gaussian
kernel in the above theorem is not degenerate. Indeed, the following
result holds true.

\begin{theorem}\label{th-legge-Z}
We have $P(Z=0)+P(Z=1)<1$ and $P(Z=z)=0$ for each $z\in (0,1)$. 
\end{theorem}

Theorems \ref{th-maggiore} and \ref{th-minore} below are the
main results of this paper. They are concerned with the fluctuations 
of $D_t(j)=Z_t(j) - Z_t$.

\begin{theorem}\label{th-maggiore}
For $1/2<\alpha\leq 1$, we have
$$
\sqrt{t}(Z_t(j)-Z_t)\stackrel{stably}\longrightarrow
{\mathcal N}\left(0, 
\frac{\left(1-\frac{1}{N}\right)}{(2\alpha-1)}(Z-Z^2)
\right).
$$
\noindent For $\alpha=1/2$, we have
$$
\frac{\sqrt{t}}{\sqrt{\ln(t)}}(Z_t(j)-Z_t)\stackrel{stably}\longrightarrow
{\mathcal N}\left(0, 
\left(1-\frac{1}{N}\right)(Z-Z^2)
\right).
$$

By Theorem \ref{th-legge-Z}, these limit Gaussian kernels are not
degenerate.

\end{theorem}

Note that, using the Cram\'er-Wold device (i.e. working with a linear
combination of the vector components), with a proof analogous to the
one of Theorem \ref{th-maggiore}, we can obtain the multivariate
version of the above convergences. Indeed, for instance, in the case
$1/2<\alpha\leq 1$, we can get
$$
\sqrt{t}{\bf D}_t:=\sqrt{t}\left[Z_t(1)-Z_t,\ldots,Z_t(N)-Z_t\right]
\stackrel{stably}\longrightarrow
{\mathcal N}\left(0, \frac{\bf V}{(2\alpha-1)}\right)
$$ 
where ${\bf V}_{j,j}=(1-1/N)(Z-Z^2)$ and ${\bf V}_{i,j}=-(Z-Z^2)/N$
for $i\neq j$. \\

Theorems \ref{th-media-proporz} and \ref{th-maggiore} state, in
particular, that $Z_t(j) - Z_t$ and $Z_t - Z$ have the same scaling
when $\alpha>1/2$. Using a result in \cite{ber-cri-pra-rig-2011} (see
also Proposition \ref{blocco} in the Appendix), these two theorems
combine yielding the following statement.

\begin{theorem}\label{th-maggiore-Z}
For $1/2<\alpha\leq 1$, we have  
$$\sqrt{t}(Z_t(j)-Z)\stackrel{stably}\longrightarrow 
{\mathcal N}\left(0, 
\left[\frac{1}{N}+
\frac{\left(1-\frac{1}{N}\right)}{(2\alpha-1)}
\right]
(Z-Z^2)
\right).
$$
\noindent 
For $\alpha=1/2$, we have 
$$
\frac{\sqrt{t}}{\sqrt{\ln(t)}}(Z_t(j)-Z)\stackrel{stably}\longrightarrow 
{\mathcal N}\left(0, 
\left(1-\frac{1}{N}\right)
(Z-Z^2)
\right).
$$
\end{theorem}

For $\alpha < \frac{1}{2}$, synchronization is slower, as $D_t(j)=Z_t(j) -
Z_t$ scales like $t^{-\alpha}$, as the following Theorem establishes.

\begin{theorem}\label{th-minore}
For $0<\alpha<1/2$, we have
$$
{\widetilde D}_t(j):=
t^{\alpha} (Z_t(j)-Z_t)\stackrel{a.s./L^1}\longrightarrow {\widetilde D}
$$
for some real random variable ${\widetilde D}$ with 
$P({\widetilde D}\neq 0)>0$.
\end{theorem}

Again, if we combine the above result with Theorem
\ref{th-media-proporz}, observing that
$$
t^{\alpha}(Z_t(j)-Z)=t^{\alpha}(Z_t(j)-Z_t)+ t^{\alpha-1/2} \sqrt{t}(Z_t-Z),
$$
we obtain the following corollary.

\begin{cor}
For $0<\alpha<1/2$, we have 
$$
t^\alpha(Z_t(j)-Z)\stackrel{P}\longrightarrow {\widetilde D}.
$$
\end{cor}

We conclude this section with a comment on the possible statistical
applications of the shown results.

\begin{remark}
\rm 
It is worthwhile to note that, since $U_t=(Z_t-Z_t^2)$  
is a strongly consistent estimator of $U=(Z-Z^2)$, from
Theorem \ref{th-media-proporz} and \ref{th-maggiore} we obtain the
convergences in distribution to the standard normal distribution
${\mathcal N}(0,1)$ of the following sequences:
\begin{itemize}
\item $\sqrt{Nt}(Z_t-Z)/\sqrt{U_t}$ for any $\alpha$,
\item $(1-1/N)^{-1/2}\sqrt{(2\alpha-1)}\sqrt{t}(Z_t(j)-Z_t)/\sqrt{U_t}$ 
for $\alpha\in (1/2,1]$,
\item $(1-1/N)^{-1/2}\sqrt{t}\ln(t)^{-1/2} (Z_t(j)-Z_t)/\sqrt{U_t}$ 
for $\alpha=1/2$.
\end{itemize}

The first of the above convergences can be used in order to provide an
asymptotic confidence interval for the limit random variable $Z$ (see
\cite{ber-cri-pra-rig-2010}). The other two can be useful in order to
construct statistical tests and asymptotic confidence intervals for
the parameter $\alpha$.
\end{remark}

\section{Proofs} 
\label{proofs}

\subsection{Proof of Theorem \ref{th-media-proporz}}

We will prove the almost sure conditional convergence (and so the
stable convergence) using Theorem 2.2 in \cite{crimaldi-2009} (see
also Prop. 1 in \cite{ber-cri-pra-rig-2011}).  To this purpose, we
observe that $(Z_t)_t$ is an $\cal F$-martingale which converges to a
random variable $Z$ a.s. and in mean and which satisfies the following two
conditions:
\begin{itemize}
\item[1)] $E\left[\,\sup_{k} \sqrt{k}|Z_{k+1}-Z_k|\,\right]<+\infty$;
\item[2)] $t\sum_{k\geq t} (Z_{k+1}-Z_k)^2\stackrel{a.s.}\longrightarrow
\frac{1}{N}(Z-Z^2)$.
\end{itemize}
Indeed, the first condition immediatly follows from  
$$
|Z_{k+1}-Z_k|=\frac{1}{m+k+1}
\left|
\frac{\sum_{i=1}^N I_{k+1}(i)}{N}-Z_k
\right|=O(k^{-1}).
$$ 

Regarding the second condition, we observe that  
\begin{equation*}
t\sum_{k\geq t} (Z_{k+1}-Z_k)^2
=t\sum_{k\geq t}\frac{1}{(m+k+1)^2}
\left(
\frac{\sum_{i=1}^N I_{k+1}(i)}{N}-Z_k
\right)^2
\end{equation*}
and so the desidered convergence  follows by Lemma \ref{lemma-kro} (with 
$a_k=1$, $b_k=k$, 
$Y_k=k^2\big(\frac{\sum_{i=1}^N I_{k+1}(i)}{N}-Z_k\big)^2/(m+k+1)^2$ 
and ${\cal G}_k={\cal F}_{k+1}$) since
$$
\sum_{k=1}^{\infty}\frac{1}{k^2}\frac{k^4E\left[\left(
\frac{\sum_{i=1}^N I_{k+1}(i)}{N}-Z_k
\right)^4\right]}{(m+k+1)^4}
\leq
\sum_{k=1}^{\infty}\frac{1}{k^2}<+\infty
$$
and (taking into account the 
${\cal F}_k$-conditional independence of the indicator functions 
$I_{k+1}(i)$ for $i=1,\ldots,N$)
\begin{equation*}
\begin{split}
&\frac{k^2}{(m+k+1)^2} 
E\left[\left(\frac{\sum_{i=1}^N I_{k+1}(i)}{N}-Z_k
\right)^2\,|\,{\mathcal F}_k\right]
\\
&=\frac{k^2}{(m+k+1)^2}
\mathrm{Var}
\left[ \frac{\sum_{i=1}^N I_{k+1}(i)}{N}\,|\,{\mathcal F}_k
\right]\\
&=
\frac{k^2}{(m+k+1)^2}\frac{1}{N^2}\sum_{i=1}^N
\mathrm{Var}[I_{k+1}(i)\,|\,{\mathcal F}_k]
\\
&=
\frac{k^2}{(m+k+1)^2}\frac{1}{N^2}\sum_{i=1}^N 
\left[\alpha Z_k+(1-\alpha)Z_k(i)-(\alpha Z_k+(1-\alpha)Z_k(i))^2\right]
\\
&\stackrel{a.s.}\longrightarrow \frac{1}{N}(Z-Z^2).
\end{split}
\end{equation*}

\subsection{Proof of Theorem \ref{th-legge-Z}}

The first part of Theorem \ref{th-legge-Z} immediately follows from
the relation $E[Z^2]<E[Z]$, that is $E[Z(1-Z)]>0$, (see Lemma 1 in
\cite{daipra-2014}).\\

\indent The second part of Theorem \ref{th-legge-Z} is a consequence
of the almost sure conditional convergence stated in Theorem
\ref{th-media-proporz}. Indeed, if we denote by $K_t$ a version of the
conditional distribution of $\sqrt{t}(Z_t-Z)$ given ${\mathcal F}_t$,
then there exists an event $A$ such that $P(A)=1$ and, for each
$\omega\in A$,
\begin{equation*}
{\textstyle\lim_t} Z_t(\omega)=Z(\omega)\quad\hbox{and}\quad
K_t(\omega)\stackrel{weakly}\longrightarrow 
{\mathcal N}\left(0,\frac{1}{N}(Z(\omega)-Z^2(\omega))\right).
\end{equation*}

Assume now, by absurb, that there exists $z\in (0,1)$ with $P(Z=z)>0$,
and set $A'=A\cap\{Z=z\}$ and define $B_t$ as the ${\mathcal
  F}_t$-measurable random set $\{\sqrt{t}(Z_t-z)\}$. Then $P(A')>0$
and, since $E\left[I_{\{Z=z\}}\,|\,{\mathcal F}_t\right]$ converges
almost surely to $I_{\{Z=z\}}$, there exists an event $A''$ such that
$P(A'')>0$, $A''\underline{\subset} A'$ and, for each $\omega\in A''$, 
\begin{equation*}
K_t(\omega)(B_t(\omega))
\!=\!E\left[I_{\{\sqrt{t}(Z_t-z)\}}\left(\sqrt{t}(Z_t-Z)\right)
\,|\,{\mathcal F}_t\right](\omega)
\!=\!E\left[I_{\{Z=z\}}\,|\,{\mathcal F}_t\right](\omega)
\!\longrightarrow\! 
I_{\{Z=z\}}(\omega)\!=\!1.
\end{equation*}
On the other hand, we observe that $N^{-1}(Z(\omega)-Z^2(\omega))$ is
not zero when $\omega\in A'$. Hence, if $D$ is the discrepancy metric
defined by
$$
D[\mu,\nu]=
{\textstyle\sup_{\{B\in\{\hbox{closed balls of }\R\}\}}} |\mu(B)-\nu(B)|,
$$ 
which metrizes the weak convergence of a sequence of probability
distributions on $\R$ in the case when the limit distribution is
absolutely continuous with respect to the Lebesgue measure on $\R$
(see \cite{gibbs-2002}), then, for each $\omega\in A'$,  we have
\begin{equation*}
\begin{split}
K_t(\omega)(B_t(\omega))&=
\left| 
K_t(\omega)(B_t(\omega))-
{\mathcal N}\left(0,\frac{1}{N}(Z(\omega)-Z^2(\omega))
\right)(B_t(\omega))
\right|
\\
&
\leq  
D\!\left[
K_t(\omega),
{\mathcal N}\left(0, \frac{1}{N}(Z(\omega)-Z^2(\omega))\right)
\right]
\longrightarrow 0.
\end{split}
\end{equation*}
This contradicts the previous fact and the proof is concluded.

\subsection{Proof of Theorems \ref{th-maggiore} and \ref{th-maggiore-Z}}

{\em Proof of Theorem \ref{th-maggiore}.} Let us set
\begin{equation*}
\begin{split}
L_0&=D_0=0\\
L_t&=D_t-\sum_{k=0}^{t-1}\left(E[D_{k+1}|{\cal F}_k]-D_k\right)\\
&=D_t+\alpha\sum_{k=0}^{t-1}\frac{D_k}{m+k+1}\\
&=D_t+\alpha\sum_{k=1}^{t-1}\frac{D_k}{m+k+1}\qquad\hbox{per } t\geq 1.
\end{split}
\end{equation*}
Then $(L_t)$ is an ${\cal F}$-martingale by construction and, for each
$t\geq 1$, we can write
\begin{equation*}
D_{t+1}=\left(1-\frac{\alpha}{m+t+1}\right)D_t+\Delta L_{t+1},
\end{equation*}  
where $\Delta L_{t+1}=L_{t+1}-L_t$. Iterating the above relation and
using the notation in Lemma \ref{lemma-coef}, we obtain
\begin{equation*}
D_{t+1}=c_{1,t}D_1+\sum_{k=1}^t c_{k+1,t}\Delta L_{k+1}.
\end{equation*}

We firstly consider the case $\alpha>1/2$. By Lemma \ref{lemma-coef},
we have $\sqrt{t} c_{1,t}\sim c t^{-(\alpha-1/2)}\to 0$ since
$\alpha>1/2$. Therefore, it is enough to prove the convergence
\begin{equation*}
\sqrt{t}\sum_{k=1}^t c_{k+1,t} \Delta L_{k+1}\stackrel{stably}
\longrightarrow {\mathcal N}\left(0, \frac{V}{2\alpha-1}\right),
\end{equation*}
where $V=(1-1/N)(Z-Z^2)$. To this purpose, let us define
$$
Y_{t,k}=\sqrt{t}\,c_{k+1,t}\Delta L_{k+1}\quad\hbox{and}\quad
{\mathcal G}_{t,k}={\mathcal F}_{k+1}.$$ 
Thus,
$\{Y_{t,k},\, {\cal G}_{t,k}: 1\leq k\leq t\}$ 
is a square-integrable martingale difference array. Indeed,
we have
$$
E[Y_{t,k}^2]<+\infty\quad\hbox{and}\quad 
E[Y_{t,k+1}|{\mathcal G}_{t,k}]=
\sqrt{t}\,c_{k+2,t}E[\Delta L_{k+2}|{\mathcal F}_{k+1}]=0.
$$
Applying Theorem \ref{hall}, the convergence 
$$
\sqrt{t}\sum_{k=1}^t c_{k+1,t} \Delta L_{k+1}=\sum_{k=1}^t Y_{t,k}\stackrel{stably}
\longrightarrow {\mathcal N}\left(0, \frac{V}{2\alpha-1}\right)
$$
is ensured if the following conditions are satisfied:
\begin{itemize}
\item[1)] $\max_{1\leq k\leq t}|Y_{t,k}|\stackrel{P}\longrightarrow 0$;
\item[2)] $E[\max_{1\leq k\leq t}Y_{t,k}^2]$ is bounded in $t$;
\item[3)] $\sum_{k=1}^t Y_{t,k}^2\stackrel{P}\longrightarrow V/(2\alpha-1)$.
\end{itemize}
Hence, in order to conclude the proof, we now verify the above three
conditions.\\

\indent {\sc Proof of condition 1)}. We observe that
\begin{equation*}
\begin{split}
\Delta L_{k+1}&=D_{k+1}-D_k+\alpha\frac{D_k}{m+k+1}\\
&=
(Z_{k+1}(j)-Z_k(j))-(Z_{k+1}-Z_k)+\alpha\frac{D_k}{m+k+1}\\
&
=\frac{I_{k+1}(j)-Z_k(j)}{m+k+1}-\frac{(\sum_{i=1}^N I_{k+1}(i)/N)-Z_k}{m+k+1}+
\alpha\frac{D_k}{m+k+1}
\end{split}
\end{equation*}
and so $|\Delta L_{k+1}|=O(k^{-1})$. Therefore, condition 1) easily
follows since\footnote{We use the notation $a_t\simeq b_t$ when $\lim_t
  a_t=\lim_t b_t$.}
\begin{equation*}
({\textstyle\max_{1\leq k\leq t}}|Y_{t,k}|)^3\leq
\sum_{k=1}^{t-1}|Y_{t,k}|^3+|Y_{t,t}|^3
\simeq   
\frac{1}{t^{3\alpha-3/2}} 
\sum_{k=1}^{t-1}\frac{k^3 O(k^{-3})}{k^{1-(3\alpha-2)}}  
+ \frac{t^3 O(t^{-3})}{t^{3/2}}\longrightarrow 0.
\end{equation*}

\indent {\sc Proof of condition 2)}. We have  
\begin{equation*}
\begin{split}
E\left[{\textstyle\max_{1\leq k\leq t}}Y_{t,k}^2\right]
&\leq\sum_{k=1}^t E[Y_{t,k}^2] 
=
t\sum_{k=1}^t c_{k+1,t}^2E[(\Delta L_{k+1})^2]\\
&=
t\sum_{k=1}^{t-1} c_{k+1,t}^2E[(\Delta L_{k+1})^2]+t E[(\Delta L_{t+1})^2]\\
&\simeq 
\frac{1}{t^{2\alpha-1}}
\sum_{k=1}^{t-1} \frac{ k^2 O(k^{-2})}{k^{1-(2\alpha-1)}}+
\frac{t^{2} O(t^{-2})}{t}.
\end{split}
\end{equation*}
The last term is bounded in $t$ since 
$$
\sum_{k=1}^{t-1} \frac{1}{k^{1-(2\alpha-1)}}
\sim \frac{t^{2\alpha-1}}{2\alpha-1}.
$$

\indent{\sc Proof of condition 3)}. We observe that
$$
\sum_{k=1}^t Y_{t,k}^2=t\sum_{k=1}^t c_{k+1,t}^2(\Delta L_{k+1})^2
\simeq \frac{1}{t^{2\alpha-1}}
\sum_{k=1}^{t-1} \frac{(k\Delta L_{k+1})^2}{k^{1-(2\alpha-1)}}
+ \frac{(t\Delta L_{t+1})^2}{t}.
$$
Moreover, we have
\begin{equation*}
\begin{split}
(\Delta L_{k+1})^2&=
(Z_{k+1}(j)-Z_k(j))^2+(Z_{k+1}-Z_k)^2+\alpha^2\frac{D_k^2}{(m+k+1)^2}\\
&
-2(Z_{k+1}(j)-Z_k(j))(Z_{k+1}-Z_k)+2\alpha\frac{(Z_{k+1}(j)-Z_k(j))D_k}{m+k+1}
-2\alpha\frac{(Z_{k+1}-Z_k)D_k}{m+k+1}\\
&=(Z_{k+1}(j)-Z_k(j))^2+(Z_{k+1}-Z_k)^2-2(Z_{k+1}(j)-Z_k(j))(Z_{k+1}-Z_k)\\
&+\alpha^2\frac{D_k^2}{(m+k+1)^2}
+2\alpha\frac{(I_{k+1}(j)-Z_k(j))D_k}{(m+k+1)^2}
-2\alpha\frac{(\sum_{i=1}^N I_{k+1}(i)/N-Z_k)D_k}{(m+k+1)^2}.
\end{split}
\end{equation*}
Since $D_k\stackrel{a.s.}\longrightarrow 0$, it is easy to get that
each of the three terms
\begin{multline*}
 \frac{1}{t^{2\alpha-1}}
\!\sum_{k=1}^{t-1} \frac{k^2D_k^2}{(m+k+1)^2k^{1-(2\alpha-1)}},
\\
\frac{1}{t^{2\alpha-1}}
\!\sum_{k=1}^{t-1}
\frac{k^2(I_{k+1}(j)-Z_k(j))D_k}{(m+k+1)^2 k^{1-(2\alpha-1)}},
\\
\frac{1}{t^{2\alpha-1}}
\!\sum_{k=1}^{t-1}
\frac{k^2(\sum_{i=1}^N I_{k+1}(i)/N-Z_k)D_k}{(m+k+1)^2 k^{1-(2\alpha-1)}}
\end{multline*}
converges almost surely to zero.  Hence, setting
$a_k=k^{2-2\alpha}$, $b_k=k^{2\alpha-1}$, ${\mathcal G}_{k}={\mathcal F}_{k+1}$ and
$$
Y_k=k^2[(Z_{k+1}(j)-Z_k(j))^2+(Z_{k+1}-Z_k)^2-2(Z_{k+1}(j)-Z_k(j))(Z_{k+1}-Z_k)],
$$
by Lemma \ref{lemma-kro}, condition 3) is satisfied if
$$
\sum_{k=1}^{\infty}\frac{E[Y_k^2]}{k^2}<+\infty\quad\hbox{and}\quad
E[Y_k\,|\,{\mathcal F}_k]\stackrel{a.s.}\longrightarrow V.
$$  
The first condition is trivialy satisfied since
$Y_k^2=k^4 O(k^{-4})$.  As concerns the second one, we have already
verified in the proof of Theorem \ref{th-media-proporz} that
$$ 
E\left[k^2(Z_{k+1}-Z_k)^2\,|\,{\mathcal F}_k\right]
=\frac{k^2}{(m+k+1)^2}
E\left[\left(\frac{\sum_{i=1}^N I_{k+1}(i)}{N}-Z_k\right)^2
\,|\,{\mathcal F}_k\right]
\stackrel{a.s.}\longrightarrow
  \frac{1}{N}\left(Z-Z^2\right).
$$
Further, we can check that
$$
E\left[k^2(Z_{k+1}(j)-Z_k(j))^2\,|\,{\mathcal F}_k\right]
=
\frac{k^2}{(m+k+1)^2}
E\left[\left(I_{k+1}(j)-Z_k(j)\right)^2\,|\,{\mathcal F}_k\right]
\stackrel{a.s.}\longrightarrow (Z-Z^2).
$$
Finally, we observe that
\begin{equation*}
\begin{split}
&E\!\left[k^2(Z_{k+1}(j)-Z_k(j))(Z_{k+1}-Z_k)\,|\,{\mathcal F}_k\right]=\\
&
\frac{k^2}{(m+k+1)^2}E\!\!\left[(I_{k+1}(j)-Z_k(j))
\left(\frac{\sum_{i=1}^N I_{k+1}(i)}{N}-Z_k\right)\,|\,{\mathcal F}_k\!\right]=\\
&
\frac{k^2}{(m+k+1)^2}E\!\!\left[
\frac{I_{k+1}(j)}{N}+\frac{\sum_{i=1,i\neq j}^N I_{k+1}(i)I_{k+1}(j)}{N}
-Z_kI_{k+1}(j)
-Z_k(j)\frac{\sum_{i=1}^N I_{k+1}(i)}{N}
+Z_kZ_k(j)\,|\,{\mathcal F}_k\!\right]
\\
&\stackrel{a.s.}\longrightarrow
Z/N+Z^2-Z^2/N-Z^2-Z^2+Z^2=Z/N-Z^2/N=\frac{1}{N}(Z-Z^2).
\end{split}
\end{equation*}
The proof of the case $\alpha>1/2$ is so concluded.\\

\indent Finally, the proof of the case $\alpha=1/2$ is essentially the
same. Indeed, we have \\ $(\sqrt{t}/\sqrt{\ln(t)}) c_{1,t}\sim
c (\ln(t))^{-1/2}\to 0$ and so we can continue the argument as above by
setting
$$
Y_{t,k}=\frac{\sqrt{t}}{\sqrt{\ln(t)}}c_{k+1,t}\Delta L_{k+1}.
$$
For this purpose, we observe that
\begin{equation*}
({\textstyle\max_{1\leq k\leq t}}|Y_{t,k}|)^3\leq
\sum_{k=1}^{t-1}|Y_{t,k}|^3+|Y_{t,t}|^3
\simeq   
\frac{1}{\ln(t)\sqrt{\ln(t)}} 
\sum_{k=1}^{t-1}\frac{k^3 O(k^{-3})}{k^{1+1/2}}   
+ \frac{t^3 O(t^{-3})}{\ln(t)\sqrt{\ln(t)}\,t^{3/2}}\longrightarrow 0
\end{equation*}
and 
\begin{equation*}
\begin{split}
E\left[{\textstyle\max_{1\leq k\leq t}}Y_{t,k}^2\right]
\leq 
\frac{1}{\ln(t)}\sum_{k=1}^{t-1} \frac{k^2 O(k^{-2})}{k}+
\frac{t^{2} O(t^{-2})}{t\ln(t)} 
\end{split}
\end{equation*}
with 
$$
\sum_{k=1}^{t-1} \frac{1}{k}
\sim \ln(t).
$$ 

Finally, we apply Lemma \ref{lemma-kro} with $a_k=k$,
$b_k=\ln(k)$ and ${\mathcal G}_k={\mathcal F}_{k+1}$, in order to
prove that
$$
\sum_{k=1}^t Y_{t,k}^2=\frac{t}{\ln(t)}\sum_{k=1}^t c_{k+1,t}(\Delta L_{k+1})^2
\simeq \frac{1}{\ln(t)}
\sum_{k=1}^{t-1} \frac{(k\Delta L_{k+1})^2}{k}
\stackrel{a.s.}\longrightarrow V.
$$
\bigskip

{\em Proof of Theorem \ref{th-maggiore-Z}.}\\
 
{\sc Case $1/2<\alpha\leq 1$.} Since we can write 
$$
\sqrt{t}(Z_t(j)-Z)=\sqrt{t}(Z_t(j)-Z_t)+\sqrt{t}(Z_t-Z)\,,
$$ 

the stated convergence follows from Theorem \ref{th-media-proporz},
Theorem \ref{th-maggiore} and Proposition \ref{blocco}. Indeed,
$\sqrt{t}(Z_t(j)-Z_t)$ is ${\mathcal F}_t$-measurable and it converges
stably and $\sqrt{t}(Z_t-Z)$ is $\bigvee_t{\mathcal F}_t$-measurable
and it converges in the sense of
the almost sure conditional convergence (and so in the sense of the
strong stable convergence) with respect to $\mathcal F$.\\

\indent
{\sc Case $\alpha=1/2$.} Since we can write 
$$
\frac{\sqrt{t}}{\sqrt{\ln(t)}}(Z_t(j)-Z)=
\frac{\sqrt{t}}{\sqrt{\ln(t)}}(Z_t(j)-Z_t)+
\frac{\sqrt{t}(Z_t-Z)}{\sqrt{\ln(t)}}\,,
$$ 

the stated convergence follows from Theorem \ref{th-media-proporz} and
Theorem \ref{th-maggiore}. Indeed, the first term converges stably to
the desidered Gaussian kernel and the second one converges in
probability to zero.

\subsection{Proof of Theorem \ref{th-minore}}

First we prove a general property.  Let us fix $j\in\{1,\ldots, N\}$
and set
$$
{\widetilde D}_t=t^{\alpha} D_t=t^{\alpha} (Z_t(j)-Z_t).
$$

\begin{lemma}\label{prop-gen}
For any $\alpha\in [0,1]$, the sequence $({\widetilde D}_t)_t$ is an $\cal
  F$-quasi-martingale (an $\cal F$-martingale when $\alpha=0$).
\end{lemma}

{\em Proof.} The statement is trivial when $\alpha=0$ 
(since we have $N$ independent P\'olya's urns). 
Therefore, let us assume $\alpha>0$. We observe that
\begin{equation*}
\begin{split}
E[{\widetilde D}_{t+1}|{\cal F}_t]-{\widetilde D}_t&=
\left[
\left(1+\frac{1}{t}\right)^{\alpha}
-1
\right]{\widetilde D}_t
-
\left(1+\frac{1}{t}\right)^{\alpha}\frac{\alpha}{m+t+1}{\widetilde D}_t\\
&=\left[\frac{\alpha}{t}+O(1/t^2)\right]{\widetilde D}_t
-
\left[\frac{\alpha}{m+t+1}+O(1/t^2)\right]{\widetilde D}_t\\
&=
\left[\frac{m+1}{t(m+t+1)}+O(1/t^2)\right]\alpha {\widetilde D}_t\\
&=O(1/t^2)\alpha {\widetilde D}_t.
\end{split}
\end{equation*}
Hence, the sequence $({\widetilde D}_t)_t$ is an $\cal F$-quasimartingale if 
$$
{\textstyle\sum_{t=1}^{+\infty}} \frac{E[\,|D_t|\,]}{t^{1+(1-\alpha)}}<+\infty.
$$ 

This last condition is obviously true when $0<\alpha<1$ (since $(D_t)_t$
is uniformly bounded). Moreover it is also satisfied when $\alpha=1$
since (\ref{rel-4}) implies $E[\,|D_t]\,]=O(t^{-1/2})$.\\

{\em Proof of Theorem \ref{th-minore}.} By Lemma \ref{prop-gen}, the
sequence $({\widetilde D}_t)_t$ is an $\cal F$-quasi-martingale.
Moreover, by (\ref{rel-4}), we have $\sup_t E[{\widetilde
    D}_t^2]<+\infty$ and so it converges a.s. and in mean to some real
random variable $\widetilde D$.

\indent In order to prove that $P({\widetilde D}\neq 0)>0$, we will prove that
$({\widetilde D}_t^2)_t$ is bounded in $L^{p}$ for a suitable
$p>1$. Indeed, this fact implies that ${\widetilde D}_t^2$ converges in mean
to ${\widetilde D}^2$ and so, by (\ref{rel-4}), we obtain
$$
E[\,{\widetilde D}^2\,]={\textstyle\lim_t} E[\,{\widetilde D}_t^2\,]=
{\textstyle\lim_t} t^{2\alpha} E[\,D_t^2\,]>0.
$$
To this purpose, we set $p=1+\epsilon/2$, with $\epsilon>0$ 
and $x_t=E[\,|D_t|^{2+\epsilon}\,]$. We 
recall that the following recursive equation holds: 

$$
D_{t+1}=\frac{m+t}{m+t+1}D_t+\frac{1}{m+t+1}\left[I_{t+1}(j)-
\frac{\sum_{i=1}^N I_{t+1}(i)}{N}\right].
$$ 

Then we can write 
\begin{equation*}
\begin{split}
x_{t+1}&=
\left(\frac{m+t}{m+t+1}\right)^{2+\epsilon} E\big[\,|D_t|^{2+\epsilon}\,\big]\\
&+(2+\epsilon)\left(\frac{m+t}{m+t+1}\right)^{1+\epsilon}
 \frac{1}{m+t+1}
E\left[\, |D_t|^{1+\epsilon} \mathrm{sgn}(D_t) 
\left(I_{t+1}(j)-
\frac{\sum_{i=1}^N I_{t+1}(i)}{N} \right) \right]\\
&+R_t
\end{split}
\end{equation*}

where $R_t=O(t^{-2})$. Now, since 
$E\left[I_{t+1}(j)-\frac{\sum_{i=1}^N I_{t+1}(i)}{N}\Big|\mathcal{F}_t \right]= 
(1-\alpha) D_t$, we have 

\begin{equation*}
\begin{split}
x_{t+1}&=
\left(\frac{m+t}{m+t+1}\right)^{2+\epsilon}E\big[\,|D_t|^{2+\epsilon}\,\big]\\
&+
(2+\epsilon)\left(\frac{m+t}{m+t+1}\right)^{1+\epsilon}
\frac{(1-\alpha)}{m+t+1}
E\left[\, |D_t|^{1+\epsilon}\mathrm{sgn}(D_t)D_t \right]+R_t
\\
&= \left[\left(\frac{m+t}{m+t+1}\right)^{2+\epsilon}+(2+\epsilon)(1-\alpha)
\frac{(m+t)^{1+\epsilon}}{(m+t+1)^{2+\epsilon}}\right]
E\left[\, |D_t|^{2+\epsilon}\,\right]
+R_t\\
&=\left(\frac{m+t}{m+t+1}\right)^{1+\epsilon}
\left[1+\frac{(1+\epsilon)}{m+t+1}-\frac{\alpha(2+\epsilon)}{m+t+1}\right]
x_t+R_t
\\
&= 
\left(1-\frac{\alpha(2+\epsilon)}{m+t+1} \right) x_t+ g(t)
\end{split}
\end{equation*}
with $g(t)=O(t^{-2})$, where, for the last equality we have used the
elementary property
$$
\left(\frac{m+t}{m+t+1}\right)^{1+\epsilon}=
\left(1-\frac{1}{m+t+1}\right)^{1+\epsilon}=
1-\frac{1+\epsilon}{m+t+1}+O(1/t^2).
$$ 

Therefore, we have that $x_t$ satisfies the difference equation
$$
x_0=0\qquad x_{t+1}=
\left(1-\frac{\alpha(2+\epsilon)}{m+t+1} \right) x_t+ g(t).
$$ 
Since, for $\epsilon>0$ sufficiently small, we have
$\alpha(2+\epsilon)<1$ and as $t\to +\infty$
\begin{equation*}
\begin{split}
\prod_{k=0}^{t-1}
\left(1-\frac{\alpha(2+\epsilon)}{m+k+1} \right)
&=\exp\left[
\sum_{k=0}^{t-1}\ln\left(1-\frac{\alpha(2+\epsilon)}{m+k+1}\right)
\right]\\
&=\exp\left[
-\alpha(2+\epsilon)\sum_{k=0}^{t-1}\frac{1}{m+k+1}+r_1(m+t)
\right]\\
&=(m+t)^{-\alpha(2+\epsilon)}r_2(m+t)
\end{split}
\end{equation*} 
with $\lim_t r_2(m+t)\in (0,+\infty)$, by means of Lemma
\ref{solution}, we finally obtain that there exists $\epsilon>0$ such
that
$$
E\big[\,|D_t|^{2+\epsilon}\,\big]=O\left(\frac{1}{t^{\alpha(2+\epsilon)}}\right)
$$
and so 
$$
\textstyle{\sup_t} E\big[\,|{\widetilde D}_t|^{2+\epsilon}\,\big]=
{\textstyle\sup_t}\, t^{\alpha(2+\epsilon)} 
E\big[\,|D_t|^{2+\epsilon}\,\big]<+\infty.
$$
\\[20pt]

\noindent {\bf Acknowledgement.} \\

\noindent 
Paolo Dai Pra and Ida G. Minelli thank Pierre-Yves Louis and Neeraja
Sahasrabudhe for interesting discussions on the topic of this paper.

Irene Crimaldi is a member of the ``Gruppo Nazionale per l'Analisi
Matematica, la Pro\-ba\-bi\-li\-t\`a e le loro Applicazioni (GNAMPA)''
of the ``Istituto Nazionale di Alta Matematica (INdAM)''. Moreover,
she acknowledges support from CNR PNR Project ``CRISIS Lab''.

\appendix

\section{Appendix: Some auxiliary results}

The following lemma slightly generalizes Lemma 2 in
\cite{ber-cri-pra-rig-2011}.

\begin{lemma}\label{lemma-kro}
Let $\cal G$ be an (increasing) filtration and $(Y_k)$ be an $\cal
G$-adapted sequence of real random variables such that
$E[Y_{k}|{\mathcal G}_{k-1}]\to Y$ a.s. for some real random variable
$Y$. Moreover, let $(a_k)$ and $(b_k)$ be two sequences of strictly
positive real numbers such that
$$ b_k\uparrow +\infty,\quad
\sum_{k=1}^{\infty}\frac{E[Y_k^2]}{a_k^2b_k^2}<+\infty.
$$
Then we have:
\begin{itemize}
\item[a)] If $\frac{1}{b_t}\sum_{k=1}^t\frac{1}{a_k}\to \gamma$
for some constant $\gamma$, then 
$
\frac{1}{b_t}\sum_{k=1}^t \frac{Y_k}{a_k}\stackrel{a.s.}
\longrightarrow \gamma Y.
$
\item[b)] If $b_t\sum_{k\geq t} \frac{1}{a_kb_k^2}\to \gamma$
for some constant $\gamma$, then 
$
b_t\sum_{k\geq t} \frac{Y_k}{a_kb_k^2}\stackrel{a.s.}
\longrightarrow \gamma Y.
$
\end{itemize}
\end{lemma}

\begin{proof}
Let us set 
$$M_t=\sum_{k=1}^t \frac{Y_k-E[Y_k|{\mathcal G}_{k-1}]}{a_kb_k}.$$
Then $(M_t)$ is a $\cal G$-martingale such that
$$
{\textstyle\sup_t} E[M_t^2]\leq 
4\sum_{k=1}^{\infty}\frac{E[Y_k^2]}{a_k^2b_k^2}<+\infty.
$$
Therefore $(M_t)$ converges almost surely and, by Kronecker's lemma, we get
$$
\frac{1}{b_t}\sum_{k=1}^t b_k \frac{Y_k-E[Y_k|{\mathcal G}_{k-1}]}{a_kb_k}
\stackrel{a.s.}\longrightarrow 0.
$$
Moreover, by Abel's lemma, we have
$$
b_t\sum_{k\geq t} \frac{Y_k-E[Y_k|{\mathcal G}_{k-1}]}{a_kb_k^2}
\stackrel{a.s.}\longrightarrow 0.
$$
This is sufficient in order to conclude since,  in case a), we have
\begin{equation*}
\frac{1}{b_t}\sum_{k=1}^t \frac{Y_k}{a_k}=
\frac{1}{b_t}\sum_{k=1}^t \frac{E[Y_k|{\mathcal G}_{k-1}]}{a_k}+
\frac{1}{b_t}\sum_{k=1}^t b_k \frac{Y_k-E[Y_k|{\mathcal G}_{k-1}]}{a_kb_k}
\stackrel{a.s.}\longrightarrow \gamma Y
\end{equation*}
and, in case b), we have
\begin{equation*}
b_t\sum_{k\geq t} \frac{Y_k}{a_kb_k^2}=
b_t\sum_{k\geq t} \frac{E[Y_k|{\mathcal G}_{k-1}]}{a_kb_k^2}+
b_t\sum_{k\geq t} \frac{Y_k-E[Y_k|{\mathcal G}_{k-1}]}{a_kb_k^2}
\stackrel{a.s.}\longrightarrow \gamma Y.
\end{equation*}
\end{proof}

Finally, we have the following technical results.

\begin{lemma}\label{lemma-coef}
Fix $\alpha\in (0,1]$ and set 
\begin{equation*}
c_{k,t}=
\begin{cases}
&\prod_{h=k}^t\left(1-\frac{\alpha}{m+h+1}\right)=
\frac{\prod_{h=1}^{m+t}\left(1-\frac{\alpha}{h+1}\right)}
{\prod_{h=1}^{m+k-1}\left(1-\frac{\alpha}{h+1}\right)}
\quad\hbox{for } 1\leq k\leq t\\
&1\quad\hbox{for } k=t+1.
\end{cases}
\end{equation*}
Then 
$$
c_{1,t}\sim c t^{-\alpha}\;\hbox{with } c\in(0,+\infty)
\quad\hbox{and}\quad
{\textstyle\lim_{k\rightarrow +\infty}\sup_{t\geq k}}
\left|\frac{c_{k,t}}{\left(\frac{k}{t}\right)^\alpha}-1\right|=0.
$$
\end{lemma}

\begin{proof}
It is enough to observe that, as $t\to +\infty$, we have
\begin{equation*}
\begin{split}
\prod_{h=1}^{m+t}\left(1-\frac{\alpha}{h+1}\right)&=
\exp\left[ \sum_{h=1}^{m+t}\ln\left(1-\frac{\alpha}{h+1}\right) \right]\\
&=
\exp\left(-\alpha\sum_{h=0}^{m+t}\frac{1}{h+1}+g_1(m+t)\right)\\
&=
\exp\left(-\alpha\ln(m+t)+g_2(m+t)\right)=\\
&=(m+t)^{-\alpha}g_3(m+t)
\end{split}
\end{equation*}
with $\lim_{h\to +\infty} g_3(h)\in (0,+\infty)$.
\end{proof}

\begin{lemma}\label{solution}
The solution of the difference equation 
$$
x_0=0 \qquad x_{t+1}=f(t)x_t+g(t)\quad\hbox{for }t\geq 0,
$$
with $f(t)>0$ for all $t\geq 0$, is given by
$$
x_0=0\qquad x_t=
\prod_{k=0}^{t-1}f(k)\,
\sum_{i=0}^{t-1}\frac{g(i)}{\prod_{k=0}^i f(k)}\quad\hbox{for } t\geq 1.
$$
\end{lemma}

\begin{proof} If we set $y_0=x_0$ and $y_t=x_t/\prod_{k=0}^{t-1}f(k)$ for 
$t\geq 1$,  then we find that $y_t$ satisfies the difference equation
$$
y_0=0\qquad y_{t+1}=y_t+F(t)\qquad\hbox{where } 
F(t)=\frac{g(t)}{\prod_{k=0}^t f(k)}.
$$ 
Hence, it is easy to verify that $y_t=\sum_{i=0}^{t-1} F(i)$ for
$t\geq 1$ and so the proof is concluded.
\end{proof}

We next state two results without proofs.

\begin{theorem}\label{hall} (Theorem 3.2 in \cite{hall-1980})\\
Let $\{S_{n,k},\, {\cal F}_{n,k}:\, 1\leq k\leq k_n , n\geq 1\}$ be a
zero-mean, square-integrable martingale array with differences $Y_{n,k}$,
and let $\eta^2$ be an a.s. finite random variable.  
Suppose that
\begin{itemize}
\item[1)] $\max_{1\leq k\leq k_n}|Y_{n,k}|\stackrel{P}\longrightarrow 0$;
\item[2)] $E[\max_{1\leq k\leq k_n}Y_{n,k}^2]$ is bounded in $n$;
\item[3)] $\sum_{k=1}^{k_n} Y_{n,k}^2\stackrel{P}\longrightarrow \eta^2$
\end{itemize}
and the $\sigma$-fields are nested, i.e.  ${\cal
  F}_{n,k}\underline\subset{\cal F}_{n+1,k}$ for $1\leq k\leq k_n,\,
n\geq 1$. Then $S_{n,k_n}=\sum_{k=1}^{k_n} Y_{n,k}$ converges stably
to a random variable with characteristic function
$\varphi(u)=E[\exp(-\eta^2 u^2/2)]$, i.e. to the Gaussian kernel ${\mathcal
  N}(0,\eta^2)$.\\
(Here, the symbol ${\cal N}(0,0)$ means the Dirac distribution $\epsilon_0$.)
\end{theorem}

\begin{prop}\label{blocco} (Lemma 1 in \cite{ber-cri-pra-rig-2011})\\
Suppose that $C_n$ and $D_n$ are $S$-valued random variables, that $M$
and $N$ are kernels on $S$, and that ${\mathcal G}=({\mathcal G}_n)_n$ 
is an (increasing) filtration satisfying for all $n$
$$
\sigma(C_n)\underline\subset{\cal G}_n\quad\hbox{and }\quad 
\sigma(D_n)\underline\subset{\cal G}_\infty=
\sigma\left({\textstyle\bigcup_n}{\cal G}_n\right) 
$$

\noindent If $C_n$ stably converges to $M$ and $D_n$ converges to $N$
stably in the strong sense, with respect to $\cal G$, then
$$
(C_n, D_n)\stackrel{stably}\longrightarrow M \times N.  
$$ 
(Here, $M\times N$ is the kernel on $S\times S$ such that $(M
\times N )(\omega) = M(\omega) \times N(\omega)$ for all $\omega$.)
\end{prop}

\end{document}